\documentclass[a4paper,notitlepage, 8pt,reqno]{article}

\usepackage[cp1251]{inputenc}
 \usepackage[T2A]{fontenc}
 \usepackage[russian]{babel}
\usepackage[all,arc,poly,2cell,curve,arrow,tips]{xy}
\usepackage{enumerate, eucal, amsthm,amsmath, amssymb}
\usepackage{graphicx}
\usepackage{mathtext}

\allowdisplaybreaks[4]

\newtheorem{theorem}{Theorem}
\newtheorem{lemma}{Lemma}

\newtheorem{prop}{Proposition}

\theoremstyle{definition}
\newtheorem{definition}{Definition}

\newtheorem{remark}{Remark}
\theoremstyle{plain}
\newtoks\thehProclaim
\newtheorem*{Proclaim}{\the\thehProclaim}

\theoremstyle{definition} \newtoks{\thehRemark}
\newtheorem*{Remark}{\the\thehRemark} 

 \renewcommand{\geq}{\geqslant}

\title{The Gelfand-Tsetlin type   base for the algebra $\mathfrak{g}_2$. }
                                            
\author{D.~V.~Artamonov}  

\begin{document}

 \maketitle

\begin{abstract}
The paper presents a construction of finite-dimensional irreducible representations of the Lie algebra $\mathfrak{g}_2$. The representation space is constructed as the space of solutions to a certain system of partial differential equations of hypergeometric type, which is closely related to the Gelfand-Kapranov-Zelevinsky systems. This connection allows for the construction of a basis in the representation. The orthogonalization of the constructed basis with respect to the invariant scalar product turns out to be a Gelfand-Tsetlin-type basis for the chain of subalgebras $\mathfrak{g}_2 \supset \mathfrak{sl}_3$.
\end{abstract}
\maketitle

\section{Introduction}

\subsection{Construction of representations}

A basic  problem in the theory of Lie algebra representations is the explicit construction of irreducible finite-dimensional representations and bases in them. The solution of this problem usually involves the following steps.

 First, a general construction of the representation space and the Lie algebra action on it is provided.

 Second, a basis is explicitly constructed in the representation space, and if possible, the Lie algebra action on the basis vectors is described.

 When discussing these steps, it is reasonable to limit ourselves to a specific class of Lie algebras. The basic case is that of simple Lie algebras. Of course, the case of $A$-series algebras is the most developed. But even in this case, there are several completely different approaches to solving this problem: the standard Gelfand-Zelevinsky bases \cite{gz}, the approach using essential signatures of Vinberg \cite{vb}, crystal bases \cite{k}, numerous geometric constructions \cite{gc}, the Gelfand-Tsetlin construction \cite{gt}.

Many of these constructions can be extended, with varying degrees of difficulty, to other classical series of simple Lie algebras and even to exceptional Lie algebras (see, for example, \cite{go}, where this is done for a construction originating from \cite{vb}).

 The Gelfand-Zetlin construction is perhaps more difficult to extend to other series than the others (see \cite{m}, which considers the case of the $B$, $C$, and $D$ series).

At the same time, it is the Gelfand-Tsetlin type bases that are most useful in applications in quantum physics. The fact is that the basis vectors in the Gelfand-Tsetlin type basis are constructed as eigenvectors of some operators, which in the language of quantum mechanics means that states are constructed that have definite values of some observables.

In the present paper, we consider the case of the Lie algebra $\mathfrak{g}_2$. The algebra $\mathfrak{g}_2$ arises as the algebra of observables in particle physics \cite{si}. There are papers \cite{sst}, \cite{sst1}, \cite{br}, \cite{br2}, and \cite{fls} where the authors construct Gelfand-Tsetlin-type bases for the case of $\mathfrak{g}_2$. One must note that in  these constructions different chains of subalgebras are used.

In the present paper, we also construct a Gelfand-Tsetlin-type basis for the algebra $\mathfrak{g}_2$ using the ideas from \cite{a1} and \cite{a2}. This construction reveals deep relations with the theory of hypergeometric functions, making it easy to find the action of the generators of $\mathfrak{g}_2$ on the basis vectors. It also opens up opportunities for studying tensor product decompositions.

Our program for constructing a Gelfand-Tsetlin-type basis is as follows. First, we construct the representation space and the action of the algebra $\mathfrak{g}_2$ on it. The representation space is constructed as the space of polynomial solutions of a system of partial differential equations.
Then, we construct a basis in the space of solutions, which automatically provides a basis for the representation. The construction of basis solutions is of an analytical nature. Remarkably, in these analytical  considerations, the construction of a basis solution depends on the choice of a basis in some lattice, but at the same time, the arbitrariness of the choice of a basis in the lattice is reduced to the arbitrariness of the choice of an order on some set of indices. This choice of order is analogous to the choice of a chain of subalgebras in the classical Gelfand-Tsetlin construction.

Despite this deep analogy, the constructed basis solutions are not Gelfand-Tsetlin-type vectors. However, there is a close relationship between the constructed basis solutions and the Gelfand-Tsetlin vectors. The Gelfand-Tsetlin basis is obtained from the basis solutions by orthogonalization with respect to a natural order.

In the analytical construction of the basis solutions, the mentioned choice of the basis in the lattice allows us to reduce the original system of partial differential equations to a hypergeometric type system, which is a deformation of the Gelfand-Kapranov-Zelevinsky hypergeometric system \cite{GG}.

The constructed realization of the finite-dimensional irreducible representation $\mathfrak{g}_2$ is called the A-GKZ realization in the work, by analogy with \cite{a1}, \cite{a2}. An important advantage of this realization is that it has an explicitly written invariant scalar product. This fact plays a key role in establishing a connection between the constructed basis solutions and the Gelfand-Tsetlin basis vectors. In addition, this product turns out to be useful in other problems, and using it, in \cite{a3}, \cite{a4}, in the case of $\mathfrak{gl}_3$, it is possible to obtain explicit and simple formulas for arbitrary $3j$ and $6j$ symbols for this algebra.

\subsection{The structure of the paper}

Section \ref{vvod} is introductory in nature. It recalls the definitions of the algebra $\mathfrak{g}_2$ and the group $G_2$, as well as facts about their finite-dimensional representations. It then recalls the construction of the functional realization of the representations. This section also contains a new result, explicitly presenting the generators of the ideal of relations between minors on the group $G_2$ constructed from $k$ arbitrary columns and the first $k$ consecutive rows.

In section \ref{agkz}, the Gelfand-Zetlin lattice for $\mathfrak{g}_2$ is introduced, and systems of partial differential equations are constructed based on it, which are called the GKZ (Gelfand-Kapranov-Zelevinsky) and A-GKZ (antisymmetrized GKZ) systems. Bases are constructed in the spaces of polynomial solutions for these systems.

In the first key section \ref{sol}, it is shown that the solution space of the A-GKZ system, equipped with a natural action of $\mathfrak{g}_2$, is a direct sum of all finite-dimensional irreducible representations, taken with multiplicity $1$. Conditions are provided that identify irreducible representations. This leads to a new realization of finite-dimensional irreducible representations, which is referred to as the A-GKZ realization in this paper. The basis solutions of the A-GKZ system are shown to be the basis vectors of the representations in the A-GKZ realization. It is also shown that the basis solutions of the GKZ system define the basis vectors in the functional realization.

The second key section \ref{solgc} explores the connection of the constructed basis in the realization of representations using the A-GKZ system and some Gelfand-Tsetlin-type basis. Also, in the introductory part of this section, there is a discussion of general approaches to the construction of the Gelfand-Tsetlin-type basis, including a new approach based on the A-GKZ system, which was implemented for classical series in \cite{a1}, \cite{a2}.

\section{ The Lie group $G_2$  and the Lie algebra  $\mathfrak{g}_2$. A functional representation of representations}

\subsection{ The Lie algebra $\mathfrak{g}_2$} \label{vvod}
\label{algbr}

The Lie algebra $\mathfrak{g}_2$ is the simplest of the exceptional simple Lie algebras. 
In this paper, we will deal with the complex version of this algebra. For this algebra, there is a realization as the algebra of differentiations of the complexified octonion algebra, as a subalgebra of $\mathfrak{o}_8$ of fixed points under an automorphism of order $3$, and a root realization \cite{post}. 
There are various constructions of the center of its universal enveloping algebra \cite{ag1}, \cite{ag2}.

So, let $\mathbb{O}$ be the complexified octonion algebra. The complex group $G_2$ is defined as the automorphism group $Aut(\mathbb{O})$ of this algebra. It is shown that each automorphism preserves a certain scalar product, so $G_2$ is a subgroup of $SO_8$ (also a complex group). Accordingly, the Lie algebra $\mathfrak{g}_2$ is defined as the algebra of differentiations $Der(\mathbb{O})\subset \mathfrak{o}_8$. There are explicit constructions of generators using the associator operation \cite{VO}. 

Now  let us discuss the construction of the algebra $\mathfrak{g}_2$ as a subalgebra of the fixed points of the diagram automorphism of order $3$. To do this, consider the Lie algebra $\mathfrak{o}_8$ in the split
realization, i.e. $\mathfrak{o}_8=<F_{i,j}>$, $i,j=-4,...,-1,1,...,4$, $F_{i,j}=E_{i,j}-E_{-j,-i}$. 

The Dynkin diagram $D_4$ has an automorphism $\sigma$ that rotates it by $\frac{2\pi}{3}$ around its center. Correspondingly, there is an automorphism $\psi$ of the algebra $\mathfrak{o}_8$. The set of its fixed points is a subalgebra $\mathfrak{g}_2\subset \mathfrak{o}_8$.

This approach gives us an expression for the roots of the system $\mathfrak{g}_2$ in terms of the roots of the system $D_4$, and the simple roots for $\mathfrak{g}_2$ are given by \cite{dla} (where $\varepsilon_i$ are unit vectors in four-dimensional space):


$$
\alpha_1=2\varepsilon_{-2}-\varepsilon_{-3}+\varepsilon_{-4},\,\,\,\alpha_2=\varepsilon_{-3}-\varepsilon_{-2} 
$$

The corresponding Cartan elements are written as follows

$$
H_{\alpha_1}= \frac{1}{3}(2F_{-2,-2}-F_{-3,-3}+F_{-4,-4}),\, \, \, H_{\alpha_2}=F_{-3,-3}-F_{-2,-2}
$$

The fundamental weights look as follows

$$
\omega_1=2\alpha_1+\alpha_2=\varepsilon_{-4}+\varepsilon_{-3},\,\,\,\omega_2=3\alpha_{1}+2\alpha_{2}=\varepsilon_{-2}+\varepsilon_{-3}+2\varepsilon_{-4}
$$

The weight lattice coincides with the root lattice \cite{fh}. It follows that all finite-dimensional irreducible representations are realized as subrepresentations in the tensor power of the standard $7$-dimensional representation $V$ (the fact is that the action of $\mathfrak{g}_2$ in the standard $8$-dimensional representation $ \mathfrak{o}_8$ annihilates the direct, so the standard $8$-dimensional representation $ \mathfrak{o}_8$ contains the $7$-dimensional representation $\mathfrak{g}_2$, which is called  the standard).

We note that the root realization of $\mathfrak{g}_2$ can be  explicitly identified with the classical realization as the algebra of
differentiations of the complexified octonion algebra \cite{d} (see also Section \ref{sg2}). Mostly in the sequel we will work with the realization of $\mathfrak{g}_2$ as the subalgebra of fixed points of the diagram automorphism.


\subsection{The functional realzation of representations}
\label{fr}

In the present  section, we use the realization of the group $G_2$ that arises from the realization of the Lie algebra $\mathfrak{g}_2$ as a subalgebra of $\mathfrak{o}_8$ consisting of the fixed points of the diagram automorphism of order $3$. At the level of Lie groups, this leads  a priori to the  inclusion $G_2\subset Spin(8)$, but this inclusion descends to the inclusion $G_2\subset SO(8)$. In this case, we implement $SO(8)$ as a group that acts in the space with coordinates $\{-4,...,4\}$ and preserves the quadratic form given by the matrix $\Omega $ with ones on the side diagonal.

 One possible way to implement the representations is through the functional approach. Consider the (complex) Lie group $G=G_2$ and the algebra of functions on it, $Fun(G)$. On this algebra, there is an action of $G$ given by

$$
Xf(g)=f(gX),\,\,\, f\in Fun(G),\,\,\, g,X\in G.
$$

Taking the corresponding infinitesimal action, one obtains that $Fun(G)$ is a representation of $\mathfrak{g}_2$. 

An arbitrary finite-dimensional irreducible representation is embedded in $Fun(G)$. 
Indeed, let $\omega_1$ and $\omega_2$ be fundamental weights, and the task is to construct an embedding in $Fun(G)$ of a finite-dimensional irreducible representation with highest weight $\alpha\omega_1+\beta\omega_2$, $\alpha,\beta\in\mathbb{Z}_{\geq 0}$.
To construct this embedding, we introduce the notation $a_{i}^{j}$, $i,j=-4,...,4$ for the matrix element function on the group $G_2$. Here, $j$ is the row number and $i$ is the column number.  Also put

\begin{equation}
\label{dete}
a_{i_1,...,i_k}:=det(a_i^j)_{i=i_1,...,i_k}^{j \in\text{the first   $k$ rows}}.
\end{equation}

The Lie algebra $\mathfrak{g}_2$, due to the embedding $\mathfrak{g}_2\subset\mathfrak{o}_8$, acts on these determinants according to the following rule. The generators of $\mathfrak{g}_2$ are expressed in terms of the matrix units $E_{i,j}$. For the matrix units, formally \footnote{the given formula does not correctly define the action of $\mathfrak{gl}_8$ on determinants, as it is not consistent with the relations between determinants on the group $G_2$} The action on determinants is defined as follows:

\begin{equation}
\label{edet1}
E_{i,j}a_{i_1,...,i_k}=\begin{cases} a_{\{i_1,...,i_k\}\mid_{j\mapsto i}}, =\text{ if} j\in\{i_1,...,i_k\},\\ 0 \text{ otherwise.} \end{cases}
\end{equation}

It defines the action of the generators of $\mathfrak{g}_2$ expressed through them\footnote{This action is already correctly defined}.

 If $\mathfrak{g}_2$ is realized as the subalgebra of fixed points of the diagram automorphism $\sigma$, then the positive root elements of $\mathfrak{g}_2$ are linear combinations of $F_{i,j}=E_{i,j}=E_{-j,-i}$ with $i<j$. It follows that $a_{-4}^{\alpha}a_{-4,-3}^{\beta}$ is a leading vector in the functional realization. A direct calculation shows that the weight of this vector is $\alpha\omega_1+\beta\omega_2$. Thus, we have specified a canonical embedding of an arbitrary finite-dimensional irreducible representation into a functional representation.


\subsection{The Lie group $G_2$. Relations between determinants}

Let's find the generators of the ideal of relations between the determinants $a_X$, $X\subset \{-4,...,4\}$. Let us  write relations of three types.

\subsubsection{The relations coming from the embedding $G_2\subset GL(8)$}

Like all minors of a  matrix built on the {\it first consecutive } rows , these determinants satisfy the Pl?cker relations,


$$
a_{i_1,...,i_p}a_{j_1,...,j_q}-\sum_{s=1}^qa_{j_s,i_2,...,i_p}a_{j_1,...,j_{s-1},i_1,j_{s+1},...,j_q}=0
$$

\subsubsection{The relations coming from the embedding $G_2\subset SO(8)$}
To derive the relations, we use the fact that there are Jacobian relations between the minors of the matrix $(a_i^j)$ and its inverse $((a^{-1})_i^j)$ \cite{gm}:

$$
a_{i_1,...,i_k}^{j_1,...,j_k}=det(a)(-1)^{\sum_{p=1}^k i_p+\sum_{q=1}^k j_q}(a^{-1})^{\widehat{i_1},...,\widehat{i_k}}_{\widehat{j_1},...,\widehat{j_k}}
$$

Here,  when one composes a minor on the right-hand side, all columns from $\{-4,...,4\}$ are taken except for $i_1,...,i_k$ and all rows from $\{-4,...,4\}$ except for $j_1,...,j_k$.

 Let $X$ be an element of the group $SO(8)$. Let's write down the relations that arise as a consequence for the minors of $X$ constructed on the {\it first consecutive} rows.
One has

\begin{align}
\begin{split}
\label{omg}
& X^t\Omega X=\Omega \Leftrightarrow X^{-1}=\Omega^{-1}X^t\Omega,\,\,\,\,\,  \Omega=(\omega_{i,j}),
\text{ where }
\omega_{i,j}=\begin{cases} +1,\,\, i=-j,\\  0 \text{ otherwise }  \end{cases}.
\end{split}
\end{align}

Note that $\Omega^{-1}=\Omega$.
So for the matrix elements in the case under consideration  one has 
\begin{equation}\label{ajj}a_{i}^j=a_{-j}^{-i}.\end{equation}

\begin{align}
\begin{split}
\label{rwo}
& a_{i_1,...,i_k}^{-4,...,-4+k-1}=(-1)^{i_1+...+i_k}(-1)^{-4+...+(-4+k-1)}(a^{-1})^{\widehat{i_1},...,\widehat{i_k}}_{\widehat{-4},...,\widehat{-4+k-1}}=\\
&=(-1)^{i_1+...+i_k}(-1)^{\frac{(-8+k-1)k}{2}}\cdot (a_{-j}^{-i})^{\widehat{i_1},...,\widehat{i_k}}_{\widehat{-4},...,\widehat{-4+k-1}}=\\&
=(-1)^{i_1+...+i_k}(-1)^{\frac{(-8+k-1)k}{2}}\cdot (a_{i}^{j})_{\widehat{-i_1},...,\widehat{-i_k}}^{\widehat{4},...,\widehat{4-k+1}}.
\end{split}
\end{align}

	

Thus we have proved.

\begin{lemma}[The Jacobi relations for determinants]
	\label{lj}
One has
	
	\begin{equation}
	\label{sb}
	a_{i_1,...,i_k}=\pm  a_{\widehat{-i_1},...,\widehat{-i_k}}
	\end{equation}	
	Here
	
	\begin{equation}
	\label{znk}\pm =(-1)^{i_1+...+i_k}(-1)^{\frac{(-8+k-1)k}{2}}.
	\end{equation} 
	
\end{lemma}

Introduce a more compact notation for determinants. If $X\subset \{-4,...,4\},$ then $a_X:=a_{i_1,...,i_k}$. In these notations \eqref{sb} is written as

$$
a_X=\pm a_{\widehat{-X}},
$$
where for  $X={i_1,...,i_k}$ one puts  $-X:=\{-i_1,...,-i_k\}$ and for  $Y=\{j_1,...,j_k\}$ one puts  $\widehat{Y}:=\{-4,...,4\}\setminus Y$.

Th following statement takes place.

\begin{lemma}[\cite{a2}]
	\label{lms}  
	
	The ideal of relations between the determinants $a_X$ of the form \eqref{dete} for the group $SO(8)$ is generated by the Pl?cker and Jacobi relations.

\end{lemma}



\subsubsection{The relations specific for  $G_2$}
\label{sg2}

Let's write down some relations. So far, we have used the realization of $G_2$ described in Section \ref{fr}, which comes from the realization of the algebra $\mathfrak{g}_2\subset \mathfrak{o}_8$ as the subalgebra of fixed points of a diagram automorphism of order $3$. It is quite straightforward to establish a connection with the realization of $G_2$ as automorphisms of $\mathbb{O}$. It is shown (see \cite{fh}) that $G_2$ in the realization described in Section \ref{algbr}, acting on the standard $7$-dimensional standard representation $V$, preserves the explicitly written skew-symmetric $3$-form. By adding to $V$ the one-dimensional space generated by the element {\bf 1}, we can introduce multiplication on $V\oplus \mathbb{C}<{\bf 1}>$ by using this $3$-form to construct the structure constants and interpreting {\bf 1} as the identity.  The constructed algebra is identified with the complexified octonion algebra. We extend the action of $G_2$ on $V$ to the action of $G_2$ on $V\oplus \mathbb{C}<{\bf 1}>$, assuming that $G_2$ acts trivially on { \bf 1 }. By construction, $G_2$ acts on the constructed algebra by automorphisms. In particular, this construction identifies $\mathbb{O}$ with the standard representation $\mathfrak{o}_8$.

In \cite{s1}, \cite{s2}, the first main theorem of invariant theory for the group $G_2$ is proved. The found invariants can be perceived as invariant tensors, then they can be defined in any coordinate system on $\mathbb{O}$ (or in general, defined in a coordinate-free way, as we do below). We will use a coordinate system on $\mathbb{O}$ , originating from the identification of $\mathbb{O}$ and the standard representation $\mathfrak{o}_8$.
Using the results of \cite{s1}, \cite{s2}, the following invariants for the $G_2$ action in $8$-dimensional space are given in \cite{ag2}:

\begin{equation}
\omega_{i_1,...,i_k}:=skew(\psi_{i_1,i_2,s_1}\psi_{s_1,i_2,s_2}...\psi_{s_{k-1},i_k,1}),
\end{equation}

where $\psi_{i,j,k}$ are the structure constants of the octonion algebra, and $skew$ is the antisymmetrization operation. The invariance of this tensor means that

$$
\sum_{i_1,..,i_k}\omega_{j_1,...,j_k}a_{i_1}^{j_i}...a_{i_k}^{j_k}=\omega_{i_1,...,i_k},
$$

where using \eqref{ajj} one gets 
\begin{equation}
\label{sxa}
\sum_{i_1,..,i_k}\omega_{i_1,...,i_k}a_{-i_1}^{j_i}...a_{-i_k}^{j_k}=\omega_{j_1,...,j_k}.
\end{equation}

Since $\omega$ is antisymmetric, taking $j_1=-4,...,j_k=-4+k-1$ gives a relation for the determinants $a_{i_1,...,i_k}$, which can be written shortly as $\sum_{X,|X|=k}\omega_X a_{-X}=\omega_{-4,...,-4+k-1} $. These relations exist for all $k>2$, but it is trivial for $k=8$. Thus, we have $5$ relations for $k=3,...,7$.

Let's add one more relation to the previous ones. If the previous relations were written in a coordinate system on $\mathbb{O}$ that identified this space with the $8$-dimensional standard representation of $\mathfrak{o}_8$ in the split realization, then we will write the following relation in standard coordinates on $\mathbb{O}$ (unit + imaginary octonions).

Every automorphism preserves ${\bf 1}\in \mathbb{O}$, so that $a_{ {\bf 1}}=1$. In the coordinates in which the relations \eqref{sxa} were written, this relation is written as some linear relation on the $a_i$.
So we have $6$ relations:

\begin{equation}
\label{ds}
\sum_{i_1,..,i_k}\omega_{i_1,...,i_k}a_{-i_1,...,-i_k}=\omega_{-4,...,-4+k-1}, k=3,...,7,\,\,\,\,\, a_{ {\bf 1}}=1.
\end{equation}

They are independent of each other and independent of the Pl?cker and Jacobi relations, due to the size of the determinants involved.

\subsubsection{The basis is ideally the ratio between the determinants}

\begin{theorem}
	The Plucker, Jacobi, and \eqref{ds} relations for $k=3,...,7$ form a basis in the ideal of relations between the determinants $a_X$ for the group $G_2$.
\end{theorem}

\proof

For each simple complex Lie group $G$ with a chosen Borel subgroup $B_G$, there is a flag variety $Fl_G=G/B_G$. Let $G\subset GL(n)$. Then the flag variety is embedded in $\mathbb{P}^N$, $N=2^n-2$, and the cone over the image of this embedding is given by equations that are relations between the determinants $a_X$ for $G$.

Since in the realization of $G_2$ from section \ref{fr} we have $G_2\subset SO(8)$, $B_{G_2}=B_{SO(8)}\cap G_2$, then the inclusion

$$
G_2/B_{G_2}\subset SO(8)/B_{SO(8)}, 
$$

There is also an inclusion for the cones $CFl_{G_2}\subset CFl_{SO(8)}$. It is known that $CFl_{SO(8)}$ is given by the Pl?cker and Jacobi equations (this follows almost immediately from Proposition \ref{lj}), since according to Lemma \ref{lms}, these equations generate the ideal of relations between determinants for the group $SO(8)$.
We have

\begin{align*}
& dim Fl_{SO(8)}=dim SO(8)-dimB_{SO(8)}=28-16=12,\\
&dim Fl_{G_2}=dimG_2-dim B_{G_2}=14-8=6
\end{align*}

Thus, in order to set $CFl_{G_2}$, it is necessary to add $6$ more equations to the equations that set $CFl_{SO(8)}$. We have $6$ linear relations \eqref{ds}. These relations are independent of the Plucker, Jacobi, and other relations. Therefore, in order to prove the theorem, it remains to verify that the variety defined by the Plucker, Jacobi, and relations \eqref{ds} is irreducible. But this immediately follows from the irreducibility of $CFl_{SO(8)},$ given by the Plucker and Jacobi relations, and Bertini's theorem \cite{haha} stating that the hyperplane sections of an irreducible manifold in the general case (which is the case in our case by the independence of the relations) is irreducible.

 So $CFl_{G_2}$ is given by the relations specified in the statement of the Theorem. Hence these relations generate an ideal of relations between determinants.

\section{The A-GKZ system and the A-GKZ model of representations for the algebra $\mathfrak{gl_8}$. The Gelfand-Tsetlin type basis.}
\subsection{$\Gamma$-series and the GKZ system in the general case}

\label{r2}

Detailed information about the $\Gamma$-series can
be found in \cite{GG}.

 Let $B\subset \mathbb{Z}^N$ be a lattice, $\gamma\in \mathbb{Z}^N$ be a
fixed vector. Define the {\it hypergeometric
 $\Gamma$-series } in the variables $z_1,...,z_N$ by the formula

\begin{equation}
\label{gmr}
\mathcal{F}_{\gamma}(z,B)=\sum_{b\in
	{B}}\frac{z^{b+\gamma}}{\Gamma(b+\gamma+1)},
\end{equation}

where $z=(z_1,...,z_N)$, the numerator and denominator use
multi-index notation

$$
z^{b+\gamma}:=\prod_{i=1}^N
z_i^{b_i+\gamma_i},\,\,\,\Gamma(b+\gamma+1):=\prod_{i=1}^N\Gamma(b_i+\gamma_i+1).
$$

Note that if at least one of the components of the vector $b+\gamma$ is negative, then the corresponding term in \eqref{gmr} is zero. As a result, the $\Gamma$-series considered in this paper will have only a finite number of terms. For simplicity, we will use factorials instead of $\Gamma$-functions. The sum of the series in \eqref{gmr}, if it converges, is known as the $A$-hypergeometric function.

$A$-hypergeometric functions satisfy a system of partial differential equations called the Gelfand-Kapranov-Zelevinsky system (GKZ), which consists of two types of equations.

{\bf 1.} 
Let $a=(a_1,...,a_N)$ be a vector orthogonal to the lattice $B$, then

\begin{equation}
\label{e1}
a_1z_1\frac{\partial}{\partial z_1}\mathcal{F}_{\gamma}+...+a_Nz_N\frac{\partial}{\partial z_N}\mathcal{F}_{\gamma}=(a_1\gamma_1+...+a_N\gamma_N)\mathcal{F}_{\gamma},
\end{equation}
it is sufficient to consider only the basis vectors of the lattice that is orthogonal to $B$.

{\bf 2.} 

Let $b\in {B}$ and $b=b_+-b_-$, where all the coordinates of the vectors $b_+$, $b_-$ are integers and non-negative. Let's select the non-zero elements in these vectors

$b_+=(...b_{i_1},....,b_{i_k}...)$,  $b_-=(...b_{j_1},....,b_{j_l}...)$. 

	\begin{align}
	\begin{split}
		\label{e2}
		&\mathcal{O}_b \mathcal{F}_{\gamma}=0,\,\,\,\mathcal{O}_b=(\frac{\partial}{\partial z })^{b_+}-(\frac{\partial}{\partial z })^{b_-},\\& (\frac{\partial}{\partial z })^{b_+}:= (\frac{\partial }{\partial
			z_{i_1}})^{b_{i_1}}...(\frac{\partial}{\partial z_{i_k}})^{b_{i_k}},\,\,\,\,\,
		(\frac{\partial}{\partial z })^{b_-}=(\frac{\partial }{\partial
			z_{j_1}})^{b_{j_1}}...(\frac{\partial }{\partial z_{j_l}})^{b_{j_l}}.
	\end{split}
\end{align}

Let us be given a basis in $B$. We will indicate what is the role of the operators $\mathcal{O}_b $ corresponding to the basis vectors $b\in {B}$. A system of partial differential equations can be identified with an ideal in the ring of differential operators generated by the operators that define the equations of the system. We will indicate a way of explicitly constructing an ideal in the space of differential operators with constant coefficients that corresponds to the system consisting of the equations \eqref{e2}.
To do this, we list some properties of the correspondence $b\in B\mapsto \mathcal{O}_b$. One has:

{\bf a) }  $-b \mapsto -\mathcal{O}_b$.

{\bf b) }   Let $b=b_{+}-b_{-}$, $c=c_{+}-c_{-}$, where $c_{\pm}\in \mathbb{Z}^N$ have only non-negative coordinates. Let the decomposition of $b+c$ into the difference of vectors with non-negative coordinates be $(b_{+}+c_{+})-(b_{-}+c_{-})$ (that is, there are no reductions in each coordinate in this equality). Then
	
	$$
	b+c\mapsto \mathcal{O}_{b+c}=(\frac{\partial}{\partial z})^{c_+}\mathcal{O}_{b}+(\frac{\partial}{\partial z})^{b_+}\mathcal{O}_{c}
	$$

{\bf c) }    Let $b=(b_{+}+u)-(b_{-}+v)$, $c=(c_{+}+v)-(c_{-}+u)$, where $c_{\pm},u,v$ have only integer non-negative coordinates. Let the decomposition of $b+c$ into the difference of vectors with non-negative coordinates be $(b_{+}+c_{+})-(b_{-}+c_{-})$ (that is, there are no more reductions in each coordinate in this equality). Then
	
	$$
	(\frac{\partial}{\partial z})^{u+v} \mathcal{O}_{b+c}=(\frac{\partial}{\partial z})^{c_+}\mathcal{O}_{b}+(\frac{\partial}{\partial z})^{b_+}\mathcal{O}_{c}
	$$

Introduce a definition.

\begin{definition}\label{porozd}
	Let there be a set of differential operators with constant coefficients. By a system {\it generated } by this set we shall mean a system of equations corresponding to the ideal $I$ generated by the set of differential operators, composed as follows. First, take all the operators from the set. Second, take all the operators lying in the ideal $I_1$ generated by them. Third, take the operators obtained from the operators from $I_1$ by dividing (if possible) by a differential monomial. All ideals are taken in the ring of differential operators with constant coefficients.
\end{definition}

Taking into account the discussion above, we can come to the following conclusion: the GKZ system is a system consisting of equations \eqref{e1} and equations of the system {\it generated } by the operators \eqref{e2} corresponding to the basis vectors of the lattice.

 Below, we will refer to the GKZ system as a set of equations of the second type only. That is, we will refer to the GKZ system as the system generated by the operators \eqref{e2}.

\subsection{The Gelfand-Tsetlin lattice associated with $\mathfrak{gl}_8$}

\label{nonstgl8}

We give a non-standard (compared to \cite{a1}) construction of the A-GKZ model for $\mathfrak{gl}_8$. To do this, consider the chain of subalgebras

\begin{align}
\begin{split}
\label{cpk}
& \mathfrak{gl}_8=<E_{i,j}>_{i,j\in\{-4,...,4\}}\supset\\&\supset \mathfrak{gl}_7=<E_{i,j}>_{i,j\neq 1}\supset\\&\supset \mathfrak{gl}_6=<E_{i,j}>_{i,j\notin \pm 1}\supset\\&\supset \mathfrak{gl}_3\oplus\mathfrak{ gl}_3=<E_{i,j}>_{i,j\in -4,2,3}\oplus <E_{i,j}>_{i,j\in -3,-2,4}\supset\\&
\supset \mathfrak{gl}_2\oplus \mathfrak{gl}_2 =<E_{i,j}>_{i,j\in -4,2}\oplus <E_{i,j}>_{i,j\in -3,-2}\supset \\&\supset \mathfrak{ gl}_1\oplus \mathfrak{gl}_1 =<E_{i,j}>_{i,j\in -4}\oplus <E_{i,j}>_{i,j\in -3}
\end{split}
\end{align}

With this non-standard chain, we associate a non-standard (compared to \cite{a1}) Gelfand-Tsetlin lattice $B_{GC}^{\mathfrak{gl}_8}$. This is a sublattice in the space $\mathbb{Z}^N$, $N=2^8-2$, whose coordinates are numbered by the eigen subsets in $\{-4,...,-1,1,...,4\}$. The lattice will be given by its generators (unlike \cite{a1}, where the Gelfand-Tsetlin lattice was given by equations). 

 Introduce the  ordering

\begin{equation}
	\label{chpor}
-4\prec 2\prec 3\prec -3\prec -2 \prec 4\prec -1 \prec 1
\end{equation}

Consider the vectors

\begin{equation}
\label{vagl}
v_{\alpha}=e_{i,X}-e_{j,X}-e_{i,y,X}+e_{j,y,X},
\end{equation}
where $i\prec j\prec y$.  Then

$$
B_{GC}^{\mathfrak{gl}_8}=\mathbb{Z}<v_{\alpha}>.
$$

\subsection{The GKZ and A-GKZ systems}
 \label{gkgl}

According to the general procedure, a system of GKZ is associated with the constructed lattice. This will be a system of equations for functions that depend on {\it independent} variables, which are numbered by their own subsets in $\{-4,...,-1,1,...,4\}$. We will denote these variables as $A_X$. It will be convenient for us to consider these variables as numbered ordered subsets of $X$, while requiring that these variables be antisymmetric under permutations of $X$. These variables do not obey any other relations.

 So, the {\it GKZ system} is a system generated in the sense of definition \ref{porozd} by equations constructed from the generators of the lattice. 
The equation corresponding to the generator \eqref{vagl} has the form

\begin{equation}
\label{gkzsys}
\big( \frac{\partial^2}{\partial A_{i,X}A_{j,y,X}}-\frac{\partial^2}{\partial A_{j,X}A_{i,y,X}}\big) \mathcal{F}=0
\end{equation}

Now we move on to the definition of the {\it A-GKZ system}.
As in \cite{a1}, \cite{a2}, a triple Plucker relation is associated with each generating.

$$
a_{i,X}a_{j,y,X}-a_{j,X}a_{i,y,X}+a_{y,X}a_{i,j,X}=0.
$$

These relations in turn are associated with the A-GKZ system, which is a system generated in the sense of definition \ref{porozd} by the equations

\begin{equation}
\label{agkzsys}
\big( \frac{\partial^2}{\partial A_{i,X}A_{j,y,X}}-\frac{\partial^2}{\partial A_{j,X}A_{i,y,X}}+\frac{\partial^2}{\partial A_{y,X}A_{i,j,X}}\big) F=0
\end{equation}




We now proceed to the description of the spaces of polynomial solutions of these systems.

 The situation with the GKZ system is simpler. 
First, we note that in order for $\mathcal{F}_{\gamma}(A)$ to be a polynomial, it is necessary and sufficient that $\gamma$ is an integer. 

\begin{lemma}\label{lsolgkz} Let $\gamma$ be an integer, then the following statements hold.
 The functions $\mathcal{F}_{\gamma}(A)$ are nonzero if and only if there is a vector in the shifted lattice $\gamma+B_{GC}^{gl_8}$ whose coordinates are all nonnegative. The nonzero functions $\mathcal{F}_{\gamma}(A)$ form a basis in the space of polynomial solutions of the GKZ system. \end{lemma}

The proof of the Lemma follows directly from the formula for the $\Gamma$-series.

Now, we will construct the basis solutions of the system \eqref{agkzsys}. We will use the constructions from \cite{a1}. We will choose a basis $v_1,...,v_k$ among the generators \eqref{vagl} (respectively, $k$ is the number of selected basis vectors). We will associate a vector with each generator
\begin{equation}
\label{ra0}
r_{\alpha}=e_{y,X}-e_{j,X}-e_{i,y,X}+e_{i,j,X},
\end{equation}

For  $t,s\in \mathbb{Z}^k$ introduce multiindex notations

$$
tv:=t_1v_1+...t_kv_k,\,\,\,sr:=s_1r_1+...+s_kr_k.
$$

Introduce  a series $F_{\gamma}(A)$:

\begin{align}
\begin{split}
\label{fgamma}
&J_{\gamma}^s(A)=\sum_{t\in\mathbb{Z}^k}\frac{(t+1)...(t+s)A^{\gamma+tv}}{(\gamma+tv)!},\,\,\, s\in\mathbb{Z}^k_{\geq0},\text{ где }\\
&(t+1)...(t+s):=\prod_{i=1}^k(t_i+1)...(t_i+s_i),\\&
F_{\gamma}(A)=\sum_{s\in \mathbb{Z}_{\geq 0}^s}\frac{1}{s!}(-1)^sJ^s_{\gamma-sr}(A)  \end{split}
\end{align}

Again, in order for $F_{\gamma}(A)$ to be a polynomial, it is necessary and sufficient that $\gamma$ be an integer. 
Consider the series $\gamma$ for integer vectors $\gamma$. In \cite{a1}, the following is shown (formally, in \cite{a1}, the vectors $v_{\alpha}$ have a special form, but the reasoning is repeated verbatim for any choice of basis among these vectors).

\begin{lemma}
 Let $\gamma$ be an integer, then the following statements hold.
 The functions $F_{\gamma}(A)$ are nonzero if and only if there is a vector in the shifted lattice $\gamma+B_{GC}^{\mathfrak{gl}_8}$ whose coordinates are all nonnegative. The nonzero functions $F_{\gamma}(A)$ form a basis in the space of polynomial solutions to the A-GKZ system. \end{lemma}

This Lemma immediately follows from the explicit formula for $F_{\gamma}(A)$.

\subsection{A-GKZ model of representations.}

Introduce the action of $\mathfrak{gl}_8$ on the variables $A_X$ by a rule similar to \eqref{edet1}:

\begin{equation}
\label{edet2}
E_{i,j}A_{i_1,...,i_k}=\begin{cases} A_{\{i_1,...,i_k\}\mid_{j\mapsto i}},\,\,\, j\in \{i_1,...,i_k\},\\ 0 \text{ otherwise } \end{cases}
\end{equation}
By Leibniz's rule, this action extends to an action of $\mathfrak{gl}_8$ on the space of polynomials in the variables $A_X$.

An important property of the A-GKZ system is as follows.

\begin{lemma}
	\label{ll8}
	The space of polynomial solutions $Sol_{AGKZ}$ of the A-GKZ system is a model of the $\mathfrak{gl}_8$ representations, that is, a direct sum of all irreducible finite-dimensional representations of $\mathfrak{gl}_8$ taken with multiplicity $1$.
\end{lemma}

\proof

Let $I_{\mathfrak{gl_8}}\subset\mathbb{C}[A]$ be the ideal generated by the relations between the determinants for the group $GL_8$. We make the substitution $A_X\mapsto \frac{\partial }{\partial A_X}$. We obtain the ideal $\bar{I}_{\mathfrak{gl_8}}\subset\mathbb{C}[\frac{\partial }{\partial A}]$. Consider its space of polynomial solutions $Sol_{\bar{I}_{\mathfrak{gl}_8}}$. Since the A-GKZ system is based on a set of relations between determinants, one has

\begin{equation}\label{vkl1}Sol_{AGKZ}\supset Sol_{\bar{I}_{\mathfrak{gl}_8}}.\end{equation}
In \cite{a1}, it is shown that $Sol_{\bar{I}_{\mathfrak{gl}_8}}$ is a model of representations, so essentially we need to show that $Sol_{AGKZ}= Sol_{\bar{I}_{\mathfrak{gl}_8}}$.

 Let's start the proof of this fact by associating a vector $v'_{\alpha}$ with each generator \eqref{vagl} according to the following principle. We order the indices $i,j,y$ according to the natural order. Let $i_1,i_2,i_3$ denote the indices $i,j,y$ arranged in ascending order: $i_1<i_2<i_3$, $\{i_1,i_2,i_3\}=\{i,j,y\}$. Let

\begin{equation}
\label{vapr}
v'_{\alpha}=e_{i_1,X}-e_{i_2,X}-e_{i_1,i_3,X}+e_{i_2,i_3,X}
\end{equation}

The lattice $\mathbb{Z}<v'_{\alpha}>$ is the Gelfand-Tsetlin lattice associated with $\mathfrak{gl}_8$ in the sense of \cite{a1}. As shown in \cite{a1}, if we specifically choose a basis among the vectors \eqref{vapr} and consider the equations \eqref{agkzsys} only for these basis generators, then the solution space of the system generated by these equations will coincide with $Sol_{\bar{I}_{\mathfrak{gl}_8}}$. We are considering the A-GKZ system, which includes all these equations, so we have

\begin{equation}\label{vkl2}Sol_{AGKZ}\subset Sol_{\bar{I}_{\mathfrak{gl}_8}}.\end{equation}

Hence $Sol_{AGKZ}= Sol_{\bar{I}_{\mathfrak{gl}_8}}$.

\endproof

The space of polynomial solutions of the A-GKZ system is called the A-GKZ model of finite-dimensional irreducible representations of $\mathfrak{gl}_8$.

\subsection{Solutions of the GKZ system and functional realization.} 

\begin{lemma} Let us take non-zero different $\Gamma$-series described in Lemma \ref{lsolgkz} and forming a basis in the space of solutions of the GKZ system. 
 
 Then, by substituting the determinants $a_X$ for the independent variables $A_X$, we obtain a basis for the functional realization.
 
\end{lemma} The proof of this statement is a verbatim repetition of the reasoning in the case of the standard Gelfand-Tsetlin lattice from \cite{a1}.
Moreover, the following formula holds (see \cite{a1})

\begin{equation}
\label{efd}
E_{i,j}\mathcal{F}(a)=\sum_{X}\sum_{s\in \mathbb{Z}_{\geq 0}^k}c_{X,s}\mathcal{F}_{\gamma-e_{j,X}+e_{i,X}+sr}(a),
\end{equation}

where $X$ runs over all possible subsets that do not contain $i,j$.






\section{GKZ and A-GKZ systems for the algebra $\mathfrak{g}_2$}
\label{agkz}

\subsection{  The Gelfand-Tsetlin lattice for the algebra  $\mathfrak{g}_2$}

Consider, as before, the space $\mathbb{C}^N$, $N=2^8-2$, whose coordinates $A_X$ are indexed by proper subsets $X\subset \{-4,...,-1,1,...,4\}$. It will again be convenient for us to consider the variables $A_X$ to be antisymmetric under the permutations of $X$.

 The Gelfand-Tsetlin lattice for the algebra $\mathfrak{g}_2$ is constructed in several steps.

{\bf 1)} First, we take the lattice $B_{GC}^{\mathfrak{gl}_8}\subset \mathbb{Z}^N$ for the algebra $\mathfrak{gl}_8$ as defined in Section \ref{nonstgl8}.
In this lattice, there is a set of generators of the form

\begin{equation}
\label{va}
v_{\alpha}=e_{i,X}-e_{j,X}-e_{i,y,X}+e_{j,y,X},
\end{equation}

where $i\prec j \prec y$, the order $\prec$ is defined in \eqref{chpor}.

{\bf 2)} The generating 

\begin{equation}
\label{ub}
u_{\beta}=e_{X}-e_{\widehat{-X}}.
\end{equation}

The lattice generated by the vectors $v_{\alpha}$, $u_{\beta}$ is called the Gelfand-Cetlin lattice for the algebra $\mathfrak{o}_8$ and is denoted by $B_{GC}^{\mathfrak{o}_8}$ \cite{a2}.

{\bf 3)} Add the generators of the form

\begin{equation}
\label{wg}
d_{\gamma}=e_{{\bf 1}} \text{ или } \sum_{X}\omega_X e_{-X}.
\end{equation} 

In this case, the vector $e_{{\bf 1}}$ is defined as follows. Let ${\bf 1}=\sum_{i=-4}^4 c_i e_i$ be the expression of the unit in the octonion algebra in terms of the coordinate vectors $e_{-4},...,e_{4}$ in the standard representation $\mathfrak{o}_8$ in the split realization, which is identified with $\mathbb{O}$, see section \ref{sg2}. Then we set $e_{{\bf 1}}=\sum_{i=-4}^4 c_i e_i$.

\begin{definition}
 The lattice generated by the vectors $v_{\alpha}$, $u_{\beta}$, $d_{\gamma}$ is called the Gelfand-Cetlin lattice for the algebra $\mathfrak{g}_2$ and is denoted by $B_{GC}^{\mathfrak{g}_2}$.
\end{definition}

\subsection{The GKZ system associated with $\mathfrak{g}_2$}

We write out explicitly the equations generating the GKZ system associated with the lattice $B_{GC}^{\mathfrak{g}_2}$, while making a slight modification to one of the equations.

\begin{align}
\begin{split}
\label{gkzg2}
&(\frac{\partial^2}{\partial A_{i,X}\partial A_{j,y,X}}-    \frac{\partial^2}{\partial A_{j,X}\partial A_{i,y,X}})\mathcal{F}=0,\\
&(\frac{\partial}{\partial A_X}-\pm\frac{\partial }{\partial A_{\widehat{-X}}})\mathcal{F}=0,\\
&\frac{\partial}{\partial A_{\bf 1}}\mathcal{F}=0,\\
&\sum_X\omega_X\frac{\partial }{\partial A_{-X}}\mathcal{F}=0.
\end{split}
\end{align}

The modification in \eqref{gkzg2} compared to the equations of the GKZ system for the lattice $B_{GC}^{\mathfrak{g}_2}$ consists in adding the sign $\pm$ from \eqref{znk} to the second equation. Despite this modification, we will call the system generated by the equations \eqref{gkzg2} the GKZ system.

The third equation is written implicitly. If $e_{{\bf 1}}=\sum_{i=-4}^4 c_i e_i$, then the equation is explicitly written as follows:

$$
\sum_i c_i \frac{\partial}{\partial A_i}\mathcal{F}=0.
$$

If there were no modification, its formal solutions would be written as $\Gamma$-series in powers of $A_X$:

\begin{equation}
\mathcal{F}_{\gamma}(A)=\sum_{x\in \gamma+B^{\mathfrak{g}_2}_{GC}} \frac{A^{x}}{x!}.
\end{equation}

Due to the presence of a modification, the solution formula also needs to be adjusted. To write the adjusted solution formula,
it will be convenient for us to introduce more detailed notation. Let $v_1,...,v_{k_1}$ be independent generating lattices of $B_{GC}^{\mathfrak{gl}_8}$ of type \eqref{va}\footnote{previously, instead of $k_1$, we used the index $k$ in the same situation, but now, for greater uniformity, we use $k_1$} (see the method of selecting them among all generating lattices in \cite{a1}). Let $u_1,...,u_{k_2}$ be independent vectors of the form \eqref{ub}, and let $d_1,...,d_{k_3}$ be independent vectors of the form \eqref{wg}, such that the selected vectors $v,u,d$ form a basis for $B_{GC}^{\mathfrak{g}_2}$. Then

\begin{equation}
\mathcal{F}^{\mathfrak{g}_2}_{\gamma}(A)=\sum_{t^i\in\mathbb{Z}^{k_i},i=1,2,3}  \big(\prod_{\kappa=1}^{k_2}  (\pm_{\kappa} 1)^{t^2_{\kappa}}  \big)  \frac{A^{\delta+t^1v+t^2u+t^3d}}{(\delta+t^1v+t^2u+t^3d)!},
\end{equation}

where $t^1v,t^2u,t^3d$ are understood in the sense of multi-index notation (for example, $t^1v:=t^1_1v_1+...+t^1_{k_1}v_{k_1}$), and $\kappa$ is an index that enumerates all possible vectors of the type $u$.

 The series $\mathcal{F}_{\gamma}(A)$ are, in general, formal solutions of the PDE. They will be polynomials for $\gamma\in\mathbb{Z}^N$. In this case, $\mathcal{F}_{\gamma}(A)$ essentially depends only on the class $\gamma$ $mod B_{GC}^{\mathfrak{g}_2}$.

 The polynomial will be nonzero if there is a vector in the shifted lattice $\gamma+B_{GC}^{\mathfrak{g}_2}$ with only nonnegative coordinates. Following the same approach as in \cite{a2}, we introduce the following definition.

\begin{definition}
	\label{dgc}
	
	Let us call the Gelfand-Tsetlin diagram for $\mathfrak{g}_2$ the class of an integer vector $\gamma$ $mod B_{GC}^{\mathfrak{g}_2}$ such that in the shifted lattice $\gamma+B_{GC}^{\mathfrak{g}_2}$ there is a vector with only non-negative coordinates.

\end{definition}

The following statement holds.

\begin{lemma}
 The basis in the space of polynomial solutions is formed by $\mathcal{F}^{\mathfrak{g}_2}_{\gamma}(A)$, where $\gamma$ $mod B_{GC}^{\mathfrak{g}_2}$ are all possible distinct Gelfand-Tsetlin diagrams.
\end{lemma}

As in the case of $\mathfrak{gl}_8$, the statement immediately follows from the explicit formula for $\mathcal{F}^{\mathfrak{g}_2}_{\gamma}(A)$.

\subsection{The A-GKZ system associated with $\mathfrak{g}_2$}

The A-GKZ system associated with the lattice $\mathfrak{g}_2$ is a system of equations generated in the sense of definition \ref{porozd} by the equations

\begin{align}
\begin{split}
\label{agkzg2}
&(\frac{\partial^2}{\partial A_{i,X}\partial A_{j,y,X}}-    \frac{\partial^2}{\partial A_{j,X}\partial A_{i,y,X}} + \frac{\partial^2}{\partial A_{y,X}\partial A_{i,j,X}}  )F=0,\\
&(\frac{\partial}{\partial A_X}-\pm\frac{\partial }{\partial A_{\widehat{-X}}})   F=0,\\
&\frac{\partial}{\partial A_{\bf 1}}   F=0,\\
&\sum_X\omega_X\frac{\partial }{\partial A_{-X}}   F=0.
\end{split}
\end{align}

The third equation should be understood in the same way as in the GKZ system \eqref{gkzg2}.
The system generated by the first type of equation is the A-GKZ system for $\mathfrak{gl}_8$ in the sense of \cite{a1}. In this work, its formal solutions are written out, from which it is possible to construct formal solutions of \eqref{agkzg2} (by analogy with how it is done in \cite{a2}). 

 We introduce auxiliary functions. To define them, we associate with each vector $v_{\alpha}$ of the form \eqref{va}, as in Section \ref{gkgl}, a vector

\begin{equation}
\label{ra}
r_{\alpha}=e_{y,X}-e_{j,X}-e_{i,y,X}+e_{i,j,X},
\end{equation}

For $s\in\mathbb{Z}^{k_1}_{\geq 0}$ put

\begin{equation}
\label{jf}
J_{\gamma}^{s,\mathfrak{g}_2}(A)=\sum_{t^i\in\mathbb{Z}^{k_i},i=1,2,3}   \big(  \prod_{\kappa}(\pm_{\kappa} 1)^{t^2_{\kappa}}    \big) \frac{(t^1+1)...(t^1+s)}{s!}\frac{A^{\delta+t^1v+t^2u+t^3d}}{(\delta+t^1v+t^2u+t^3d)!},
\end{equation}

where we have again used the multi-index notation. Then the formal solution of \eqref{agkzg2} is a series

\begin{equation}
\label{ff}
F_{\gamma}^{\mathfrak{g}_2}(A)=\sum_{s\in\mathbb{Z}^{k_1}_{\geq 0}}\frac{(-1)^s}{s!}J_{\gamma-sr}^{s,\mathfrak{g}_2}(A).
\end{equation}
The fact that this series satisfies the first equation of the system \eqref{agkzg2} is actually verified in \cite{a2}. The fact that the other equations are also satisfied is verified by direct inspection. In this case, the rule for differentiating these series plays a significant role.

$$
\frac{\partial}{\partial A_X }F_{\gamma}^{\mathfrak{g}_2}(A)=F_{\gamma-e_X}^{\mathfrak{g}_2}(A).
$$
The series $F^{\mathfrak{g}_2}_{\gamma}(A)$ are, in general, formal solutions of the GKZ. They will be polynomials for $\gamma\in\mathbb{Z}^N$. However, $F^{\mathfrak{g}_2}_{\gamma}(A)$ depends on the choice of a representative in the class $\gamma$ $mod B_{GC}^{\mathfrak{g}_2}$.

The polynomial solution will be nonzero if the class $\gamma$ $mod B_{GC}^{\mathfrak{g}_2}$ is a Gelfand-Tsetlin diagram
Proceeding as in \cite{a1}, the following statement is proved:

\begin{lemma}
 The basis in the space of polynomial solutions is formed by $F^{\mathfrak{g}_2}_{\gamma}(A)$, where $\gamma$ $mod B_{GC}^{\mathfrak{g}_2}$ are all possible different Gelfand-Tsetlin diagrams.
\end{lemma}

\section{Representation model for the algebra $\mathfrak{g}_2$}
\label{sol}

Introduce the action of the algebra $\mathfrak{g}_2$ on the variables $A_X$. To do this, the generators of $\mathfrak{g}_2$ are expressed in terms of matrix units, and then the formula \eqref{edet2} is used to act on the variables $A_X$ by the matrix units $E_{i,j}$. By Leibniz's rule, it extends to an action of $\mathfrak{g}_2$ on the space of polynomials in the variables $A_X$.

 The main goal of this section is to prove the following Theorem.
 
\begin{theorem}
	\label{t2}
The space of polynomial solutions of the A-GKZ system for the algebra $\mathfrak{g}_2$ is a direct sum of all finite-dimensional irreducible representations of $\mathfrak{g}_2$.
\end{theorem}

We will call such a direct sum A-GKZ a model of representations.

\proof
The proof of the Theorem consists of the following three steps. {\bf 1)} It is proved that the space of polynomial solutions is invariant under the action of $\mathfrak{g}_2$, which means that this space of solutions is a representation of $\mathfrak{g}_2$. {\bf 2)} A canonical embedding of all finite-dimensional irreducible representations into the space of solutions is constructed. {\bf 3)} It is shown that the direct sum of all canonical embeddings coincides with the solution space.

 Let's start implementing this program.

\subsection{Invariance of the solution space of the A-GKZ system}

\begin{lemma}
 \label{prdl}
 The set of solutions of the A-GKZ system is $\mathfrak{g}_2$-invariant.
\end{lemma}
\proof

Let $I_{\mathfrak{g}_2}\subset \mathbb{C}[A]$ be the ideal of relations between the determinants $a_X$ on the group $G_2$.
Define the ideal $\bar{I}_{\mathfrak{g}_2}$ in the ring of differential operators with constant coefficients $\mathbb{C}[\frac{\partial}{\partial A_X}]$ as the ideal obtained from $I_{\mathfrak{g}_2}$ by the substitution

$$
A_X\mapsto  \frac{\partial }{ \partial  A_X}.
$$

The action of $\mathfrak{g}_2$ on the variables $A_X$ defines the action of $\mathfrak{g}_2$ on $\mathbb{C}[\frac{\partial}{\partial A_X}]$.
At the same time, since the ideal $I_{\mathfrak{g}_2}$ is invariant, then the ideal $\bar{I}_{\mathfrak{g}_2}$ is also invariant, and therefore its space of polynomial solutions, denoted as $Sol_{\bar{I}_{\mathfrak{g}_2}}$, is also invariant.

 We will show that $Sol_{\bar{I}_{\mathfrak{g}_2}}$ coincides with the space of polynomial solutions of the A-GKZ system.
In Lemma \ref{lms}, it was proved that the ideal $I_{\mathfrak{g}_2}$ is generated by the ideal $I_{\mathfrak{gl}_8}$, the Jacobi relations, and the relations specific to $\mathfrak{g}_2$. Therefore,

$$
Sol_{\bar{I}_{\mathfrak{g}_2}}=Sol_{\bar{I}_{\mathfrak{gl}_8}} \cap Sol_{  Jacobi}\cap Sol_{spc\,\, \mathfrak{g}_2},
$$

where $Jacobi$ is a set of second-type equations for the A-GKZ system, and $spc\,\, \mathfrak{g}_2$ are the remaining equations of the A-GKZ system. For the case of series $A$, it was shown in the proof of Lemma \ref{ll8} that $Sol_{\bar{I}_{\mathfrak{gl}_8} }$ coincides with the space of polynomial solutions of the A-GKZ system for series $A$. But the A-GKZ system for $\mathfrak{g}_2$ is generated by the union of the A-GKZ system for the $A$ series, the Jacobi equations, and the equations that are constructed using the $G_2$-specific relations. Therefore, $Sol_{\bar{I}_{\mathfrak{g}_2}}$ indeed coincides with the space of polynomial solutions of the A-GKZ system for $\mathfrak{g}_2$. In turn, this means that the space of polynomial solutions of the A-GKZ system for $\mathfrak{g}_2$ is invariant under $\mathfrak{g}_2$.

\subsection{Canonical embedding of finite-dimensional irreducible representations}

Let $\alpha \omega_1+\beta \omega_{2}$, where $\omega_1$, $\omega_2$ are fundamental weights, be the highest weight of the representation $\mathfrak{g}_2$. In section \ref{fr}, it was discussed that in the functional realization, $a_{-4}^{\alpha}a_{-4,-3}^{\beta}$ will be the highest vector of this highest weight. This consideration will allow us to explicitly present the polynomials in the variables $A_X$ that are the highest vectors of a given highest weight and lie in the space of solutions to the A-GKZ system, which will provide us with a canonical embedding of the corresponding finite-dimensional irreducible representation into the space of polynomial solutions to the A-GKZ system.



\begin{lemma}
	\label{ostnach}
	In the solution space of the A-GKZ system, there is a polynomial
 \begin{equation}
 \label{stv1}
 (A_{-4}+A_{-4,...,3})^{\alpha}(A_{-4,-3}+A_{-4,...,2})^{\beta}
 \end{equation}
 which is the highest weight vector $\alpha\omega_1+\beta\omega_{2}$.
\end{lemma}
\proof

When $X=\{-4\}$ or $X=\{-4,-3\}$, the sign \eqref{znk} is equal to $+1$. Then the desired statement is proved directly by checking that this function is a solution of the A-GKZ for $\mathfrak{g}_2$ and that it is a highest vector with the desired highest weight.

\endproof

\subsection{The minimality of the solution space.}

\begin{lemma} The solution space of the A-GKZ system and the representation model constructed in the proof of Lemma \ref{ostnach} are the same.
\end{lemma}
For the direct sum in $\mathbb{C}[A]$ of subrepresentations with leading vectors \eqref{stv1}, we introduce the temporary notation $Mod$, and for the solution space of the A-GKZ system, we introduce the notation $Sol_{AGKZ}$.

\begin{proof}
	It has already been proven that $Mod\subset Sol_{AGKZ}$.
 Assume the opposite: $Mod\neq Sol_{AGKZ}$.
 When the independent variables $A_X$ are substituted by the determinants $a_X$, the representation $Sol_{AGKZ}$ is mapped to the Zhelobenko model, and $Mod$ is isomorphically mapped to the Zhelobenko model for reasons of irreducibility \footnote{This follows from the fact that the space of the Zhelobenko model coincides with the space of functions that can be written as polynomials in determinants that satisfy conditions of the type of homogeneity, see \cite{zh}}. From this, we can conclude (for example, by considering finite-dimensional subspaces in $Sol_{AGKZ}$ with a fixed degree of homogeneity in $A_X$ with $|X|=1$ or $7$ and in $A_X$ with $|X|=2$ or $6$), that the kernel of this mapping is non-trivial under the assumption in question. This means that there exists a non-zero function $f\in Sol_{AGKZ}$
	
	$$
	f\mid_{A_X\mapsto a_X}=0,
	$$
	
that is, $f(A_X)\in I_{\mathfrak{g}_2}$. But then $f(\frac{\partial}{\partial A})g(A)=0$ for all $g\in Sol_{AGKZ}$. In particular, $f(\frac{\partial}{\partial A})f(A)=0$, but this is only possible if $f=0$. Contradiction.
\end{proof}

The Theorem \ref{t2} is proved.

\subsection{A-GKZ model}

Let us add to the statement proved in Theorem \ref{t2} the conditions that select an irreducible representation of a fixed highest weight.

\begin{theorem}
The space of polynomial solutions of the A-GKZ system such that the homogeneous degree of $A_X$, $|X|=1$ or $7$, is equal to $\alpha$, the homogeneous degree of $A_X$, $|X|=2$ or $6$, is equal to $\beta$, and there are no $A_X$ with other values of $|X|$, is an irreducible representation of $\mathfrak{g}_2$ of highest weight $\alpha\omega_1+\beta\omega_2$.
	
At the same time, the basis of the representation consists of the basis solutions $F_{\gamma}(A)$, where $\gamma$ are the various Gelfand-Tsetlin diagrams in the sense of definition \ref{dgc}, such that
 
 \begin{align*}
 &\sum_{X:|X|=1,7}\gamma_X=\alpha,\,\,\,\sum_{X:|X|=2,6}\gamma_X=\beta,
 \end{align*}
 and the coordinates $\gamma_X$ with other values of $|X|$ are zero.
\end{theorem}
Due to the independence of the variables $A_X$, we can introduce an invariant scalar product on the polynomials $f(A)$, $g(A)$ using the formula

\begin{equation}
\label{skp}
<f(A),g(A)>:=f(\frac{\partial }{ \partial A}) g(A)\mid_{A=0},
\end{equation}

where $f(\frac{\partial }{\partial A}) $ is the result of substituting the differential operators $\frac{\partial }{\partial A_X}$ for the variables $A_X$ in $f$.

\section{Gelfand-Tsetlin-type basis for $\mathfrak{g}_2$}
\label{solgc}

In this section, we construct a Gelfand-Tsetlin-type basis for the algebra $\mathfrak{g}_2$. The construction of such a basis begins with the selection of a decreasing chain of subalgebras.

 Next one usually  investigates the branching of an irreducible representation when the  algebra is restricted according to this chain. To construct the basis, it is important to understand which irreducible representations and with what multiplicity the irreducible representation of the initial algebra decomposes under the given restriction. 
 There are completely different approaches to this problem. It is solved by adding some intermediate terms to the chain of subalgebras \cite{sht}, by constructing an action on the multiplicity space of the Yangian \cite{m}, and so on.

In this paper, we use a completely different approach, developed in \cite{a1}, \cite{a2} for classical series. The A-GKZ model is used, that is, the space of polynomial solutions of the A-GKZ system is taken as the representation space. According to the chain of subalgebras, the order on the coordinates is built in some way. Using this order, we obtain an interpretation of the A-GKZ system as a deformation of some GKZ system. Using this connection, we construct a basis in the solution space of the A-GKZ system\footnote{This has already been done in section \ref{agkz}}. It is shown that the orthogonalization of this basis is a Gelfand-Tsetlin-type basis. We arrive at this conclusion by proving that these vectors are eigenvectors of the Gelfand-Tsetlin subalgebra.

Unfortunately, this scheme is not expressed in the form of Theorems that work for all series at once. For each individual series, it is implemented "by hand". Nevertheless, it works for the series $A$, $B$, $C$, $D$, and for $G_2$.

 So, let's implement this scheme for $\mathfrak{g}_2$.

As mentioned, the construction of a Gelfand-Tsetlin-type basis for $\mathfrak{g}_2$ is based on the choice of a decreasing chain of subalgebras in $\mathfrak{g}_2$. Different chains of subalgebras are used in different works (see the discussion in \cite{sst}). We will use the simplest chain $\mathfrak{g}_2\supset \mathfrak{sl}_3$, where $\mathfrak{sl}_3$ is generated by the root elements of $\mathfrak{g}_2$ corresponding to long roots. 
A Gelfand-Tsetlin type basis can be defined as a self-basis for the Gelfand-Tsetlin subalgebra $GT\subset U(\mathfrak{g}_2)$, generated by the centers $Z(U(\mathfrak{g}_2))$, $Z(U(\mathfrak{sl}_3))\subset U(\mathfrak{g}_2)$. Let us describe a certain system of generators of these algebras.

\subsection{Subalgebra $\mathfrak{sl}_3$. Center of the universal enveloping algebras $\mathfrak{g}_2$ and $\mathfrak{sl}_3$}
\label{sug}

The algebra $\mathfrak{g}_2$ is realized as a subalgebra of $\mathfrak{o}_8$ that is invariant under the diagram automorphism of order $3$. If we denote this automorphism as $\psi$, then there is a natural linear map

$$
h: \mathfrak{o}_8\rightarrow \mathfrak{g}_2:\,\,\,f\mapsto \frac{1}{3}(f+\psi(f)+\psi^2(f)).
$$

Put

$$
D_{i,j}:=h(F_{i,j}),\,\,i,j=-4,...,4.
$$

The elements $D_{i,j}$ are root elements of $\mathfrak{g}_2$.
Take the subalgebra $\mathfrak{sl}_3$ generated by the elements corresponding to the positive simple roots $\varepsilon_{1}-\varepsilon_{2}$, $\varepsilon_{2}-\varepsilon_{3}$ (in the usual realization of the root system $A_2$), that is, by the elements \begin{equation}\label{sl3g2}D_{-3,-2}=F_{-3,-2},D_{-4,2}=F_{-4,2}\in \mathfrak{g}_2\subset \mathfrak{o}_8.\end{equation}

The center $Z(U(\mathfrak{g}_2))$ is generated (see the discussion of various generators of the center in \cite{ag1}) by Casimir elements of orders $2$ and $6$, which are written as follows:

\begin{align*}
& C^{\mathfrak{g}_2}_2=\sum_{i_1,i_2=-4}^4 D_{i_1,i_2}D_{i_2,i_1},\\
& C^{\mathfrak{g}_2}_6=\sum_{i_1,...,i_6=-4}^4 D_{i_1,i_2}D_{i_2,i_3}D_{i_3,i_4}D_{i_4,i_5}D_{i_5,i_6}D_{i_6,i_1}
\end{align*}

The center $Z(U(\mathfrak{sl}_3))$ is generated by Casimir elements of orders $2$ and $3$. If this algebra is realized in the standard way as $<E_{i,j}>_{i,j,=1,2,3}$, then they are written as follows:

\begin{align*}
& C^{\mathfrak{sl}_3}_2=\sum_{i_1,i_2=1}^3 E_{i_1,i_2}E_{i_2,i_1},\\
& C^{\mathfrak{sl}_3}_3=\sum_{i_1,i_2,i_3=1}^3 E_{i_1,i_2}E_{i_2,i_3}E_{i_3,i_1}.
\end{align*}

Accordingly, in our case, we must replace $E_{1,2}$ by $D_{-3,-2}$ and so on according to \eqref{sl3g2}.
Since $F_{i,j}^*=F_{j,i}$, it is clear from these explicit expressions that these operators are self-adjoint.

\subsection{The basis $F_{\gamma}(A)$ and the restriction procedure $\mathfrak{gl}_8\downarrow \mathfrak{gl}_3\oplus\mathfrak{gl}_3$}\label{mxm}
Consider the chain of subalgebras \eqref{cpk}. Let $g$ be one of the subalgebras in this chain.

Introduce the operation
$$
\gamma\mapsto \gamma_{max}^{g}
$$
of $g$-maximization on vectors. It can be defined in two equivalent ways..


{\bf 1.} Take the monomial $A^{\gamma}=\prod_{X}A_X^{\gamma_X}$. Introduce the variable substitution $A_X\mapsto A_{X_{max}}$ as follows. Associate with $X=\{i_1,...,i_t\}$ the tensor $e_{i_1}\wedge...\wedge e_{i_t}$. Let's act on this tensor with the raising operators from $g$, so that the result will be, up to a constant, a $g$-highest vector of the form 
\footnote{ For an arbitrary choice of the subalgebra $g\subset \mathfrak{gl}_{8}$, the resulting $g$-highest vector is not necessarily a decomposable tensor, but for the subalgebras in the chain \eqref{cpk}, this fact holds. }  $e_{j_1}\wedge...\wedge e_{j_t}$. Then $X_{max}=\{j_1,...,j_t\}$. Let

$$
A^{\gamma_{max}^g}:=\prod_{X} A_{X_{max}}^{\gamma_X}.
$$

{\bf 2.} There is an action of the algebra $\mathfrak{gl}_n$ on the space $\mathbb{C}^N$ with coordinates that are numbered subsets of $X$. This action is defined by identifying $\mathbb{C}^N$ with the exterior algebra of the standard representation of $\mathfrak{gl}_n$ (In what follows, we will use the term "action of $g$ on $\gamma$" for $\gamma\in\mathbb{C}^N$). Then, for $\gamma\in\mathbb{C}^N$, define $\gamma_{max}^g$ as the nonzero $g$-maximal vector obtained by applying the raising operators from $g$ to $\gamma$.

The maximization procedure, as can be seen from the first definition, is a linear operation, in particular

$$
(\gamma+\delta)^g_{max}=\gamma^g_{max}+\delta^g_{max}.
$$

The following statement holds.

\begin{prop}
	\label{mxmf}
	The function  $\mathcal{F}_{\gamma_{max}^g}(a)$  is a $g$-highest vector in the functional realization
\end{prop}
\proof

One has 
\begin{equation}
\label{fla}
\mathcal{F}_{\gamma_{max}^g}(a)=\sum_{t\in\mathbb{Z}^k}\frac{a^{\gamma_{max}^g+tv}}{(\gamma_{max}^g+tv)!}.
\end{equation}

Let's take one of the terms $\frac{a^{\gamma_{max}^g+tv}}{(\gamma_{max}^g+tv)!}$. The vectors \eqref{vagl} have the property that $(v_{\alpha})_{max}^g=0$ or $v_{\alpha}$.. 

In the first case, $v_{\alpha}$ contains vectors $e_{X}$ with index sets $X$ that are not maximal with respect to $g$ (when $X=\{i_1,...,i_k\}\leftrightarrow e_{i_1}\wedge...\wedge e_{i_k}$). From the structure of vector \eqref{vagl}, it is clear that such vectors $e_X$ are included in vector $v_{\alpha}$ both with positive and negative coefficients. But the vector $\gamma_{max}^g$ has no such non-zero components. So, when we add such $v_{\alpha}$ to $\gamma_{max}^g$, we get a vector with negative coordinates. The corresponding term $\frac{a^{\gamma_{max}^g+tv}}{(\gamma_{max}^g+tv)!}$ is zero.

So we can assume that \eqref{fla} contains only terms where the shift is by the vectors $v_{\alpha}$, such that $(v_{\alpha})_{max}^g=v_{\alpha}$.
In this case, the exponent vector $\gamma_{max}^g+tv$ is $g$-maximal. This means that the expression $\frac{a^{\gamma_{max}^g+tv}}{(\gamma_{max}^g+tv)!}$ is a $g$-highest vector.

\endproof

\begin{definition}
 Let $V^g_{\gamma_{max}^g}$ be the irreducible $g$-representation in the functional realization generated by the highest vector $\mathcal{F}_{\gamma_{max}^g}$.
\end{definition}

Next, we will assume that $g=\mathfrak{gl}_3\oplus \mathfrak{gl}_3$. To make the notation less cumbersome, we will omit the upper index in the notation of the maximized vector.
Using formula \eqref{efd}, one can give a definition.

\begin{equation}
	\label{dlin}
V^g_{\gamma_{max}}\subset \bigoplus <\mathcal{F}_{\gamma_{max}-e_{i_1,X_1}-...-e_{i_p,X_p}+
	e_{j_1,X_1}+...+e_{j_p,X_p}+sr}>,
\end{equation}
where the direct sum is taken over all $s\in \mathbb{Z}^k_{\geq 0},$ over all possible $p\in\mathbb{Z}_{>0},$ over all possible sets of pairs of indices $j_t\prec i_t$, $t=1,...,p$, where both indices simultaneously lie in $\{-4,2,3\}$ or $\{-3,-2,4\},$ and over all possible subsets $X_1,...,X_p$.

From \eqref{dlin}, we obtain the following inclusion

\begin{equation}
\label{fvl}
V^g_{\gamma_{max}}\subset \bigoplus_{s\in \mathbb{Z}^k_{\geq 0},\,\,\delta :\,\,\, \delta_{max}=\gamma_{max}} <\mathcal{F}_{\delta+sr}>,
\end{equation}

\begin{prop}
	\label{fvgkz}
	\begin{equation}\label{fvff}\mathcal{F}_{\omega}(a)\in \bigoplus_{s\in\mathbb{Z}^k_{\geq 0}} V^g_{(\omega-sr)_{max}}. \end{equation}
\end{prop}

\begin{proof}
	
From \eqref{fvl},  one has  that $\mathcal{F}_{\omega}(a)\in V^g_{\gamma_{max}}$ if
	
	$$
	\omega=\delta+sr,\,\,\, \delta_{max}=\gamma_{max}
	$$
	
This relationship immediately leads to the desired statement.
	
\end{proof}

Now we will prove an analogue of Proposition \ref{fvgkz} for the vector $F_{\omega}(A)$ in the A-GKZ realization.

\begin{prop}
	\begin{equation}\label{fv}F_{\omega}(A)\in \bigoplus_{s\in\mathbb{Z}^k_{\geq 0}} V^g_{(\omega-sr)_{max}} \end{equation}
\end{prop}

\begin{proof}

We will use the fact that in the A-GKZ realization, there is an explicitly written scalar product \eqref{skp}.

 Let's take the vector $\mathcal{F}_{\delta}(a)$ in the functional realization. The transition from the A-GKZ realization to the functional realization is performed by replacing $A_X\mapsto a_X$. Therefore, this vector in the A-GKZrealization is written as 
 $$
 \mathcal{F}_{\delta}(A)+pl(A),
 $$
 where $pl(A)$ belongs to the ideal of relations between determinants.

 Let us calculate the scalar product
	
	$$
	<\mathcal{F}_{\delta}(A)+pl(A), F_{\omega}(A)>=\big(\mathcal{F}_{\delta}(\frac{\partial   }{\partial  A})+pl(\frac{\partial  }{\partial  A}) \big)F_{\omega}(A)\mid_{A=0},
	$$
but since $F_{\omega}(A)$ is a solution of the A-GKZ system, then $pl(\frac{\partial }{\partial A}) F_{\omega}(A)=0$, so that
	
	$$
	<\mathcal{F}_{\delta}(A)+pl(A), F_{\omega}(A)>=<\mathcal{F}_{\delta}(A), F_{\omega}(A)>
	$$
	
From the consideration of the carriers of the functions (that is, the sets of exponents of the monomials involved in their expansions), it is clear that if the last scalar product is nonzero, then $\delta=\omega-sr$ $mod B_{GC}^{\mathfrak{gl}_8}$. 
 
 From this, we see that
 $F_{\omega}\in \bigoplus_{s\in \mathbb{Z}^k_{\geq 0}}<\mathcal{F}_{\omega-sr}>$. By comparing this ratio with Proposition \ref{fvgkz}, we obtain the desired statement.
	
\end{proof}

Let's prove the following statement.

\begin{prop}
 \label{pdbr}
 The representations on the right-hand side of formula \eqref{fv} are orthogonal with respect to the scalar product \eqref{skp} in the A-GKZ realization.\end{prop}
\begin{proof}
 
 We will prove that the representations $V^g_{(\omega-sr)_{max}}$ either coincide (if $(\omega-sr)_{max}=\omega_{max}$) or have different $g=\mathfrak{gl}_3\oplus \mathfrak{gl}_3$-highest weights.

This will be enough to prove the orthogonality. Indeed, the $\mathfrak{gl}_3\oplus \mathfrak{gl}_3$-highest weight is uniquely determined by the set of eigenvalues for the action of the Gelfand-Zetlin subalgebra constructed from the chain \eqref{cpk}\footnote{a subalgebra in $U(\mathfrak{gl}_8)$ generated by the centers of the universal enveloping algebras for all subalgebras in the chain \eqref{cpk}}. The generators of this subalgebra are self-adjoint for reasons similar to those in Section
 \ref{sug}. Therefore, if the $g=\mathfrak{gl}_3\oplus \mathfrak{gl}_3$-highest weights of the two terms in \eqref{fv} are different, then the sets of eigenvalues for the Gelfand-Zetlin subalgebra are different, and hence the terms are orthogonal.
	
	So, let's start proving the distinctness of the $g=\mathfrak{gl}_3\oplus \mathfrak{gl}_3$-highest weights. Note that $(\omega-sr)_{max}=(\omega)_{max}-r_{max}$. The vector $r$ has the form \eqref{ra}:
	
	\begin{equation}
	e_{y,X}-e_{j,X}-e_{i,y,X}+e_{i,j,X},\,\,\, i\prec j\prec y,
	\end{equation}
	
where the order $\prec$ was defined in \eqref{chpor}.
 Introduce a partial order $\prec'$ defined as follows:
	
	$$
	i\prec' j\Leftrightarrow i\prec j,\,\,\, i,j\in \{-4,2,3\}\text{ или  }i,j\in\{-3,-2,4\}.
	$$
	
	Then the following statements hold about $r_{max}$.
 
{\bf 1)} If $ j\prec' y$ (and what about $i$ is irrelevant), then $r_{max}=0$.
 
{\bf 2)} If all three indices are incomparable, then $r_{max}=r'$ is another vector of type $r$.
 
{\bf 3)}If it is not true that $j\prec'y$, but it is true that $i\prec'j$, then $
	r_{max}$  has the following description. Let the result of maximizing the set $X$ be the set of indices $X'$. Also let if after that we maximize $i$ or $j$ separately, then we get the index $t_1$ (since $i,j$ lie in one of the sets $\{-4,2,3\}$ or $\{-3,-2,4\}$, then they are maximized separately into the same index), and if simultaneously - then the indices $t_1$, $t_2$. Finally, let $y$ be maximized in $y'$. Then
	
	\begin{equation}
	\label{rmax}
	r_{max}=e_{y',X'}-e_{t_1,X'}-e_{t_1,y',X'}+e_{t_1,t_2,X'}
	\end{equation}

It is directly verified that adding vectors $r_{max}\neq 0$ changes the $\mathfrak{gl}_3\oplus \mathfrak{gl}_3$ highest weight. We will carry out this proof for $r_{max}$ from case {\bf 3)}, for $r_{max}=r'$ from case {\bf 1)}, the considerations are similar. 
 
 So we prove that the highest  weights of the vectors $\omega_{max}$ and $\omega_{max}-r_{max}$, where $r_{max}$ has the form \eqref{rmax}, are different. For definiteness, we assume that $t_1,t_2\in \{-4,2,3\}$. Then $y'\notin \{-4,2,3\}$. Due to the existence of the index set $t_1,t_2,X'$, $X'$ itself contains at most one element from $\{-4,2, 3\}$.
 
 The highest weight with respect to $\mathfrak{gl}_3\oplus \mathfrak{gl}_3$ is the union of the highest weights of these two direct summands. If $X'$ does not contain elements from $\{-4,2,3\}$, then when $r_{max}$ is subtracted, the highest weight with respect to the first summand $\mathfrak{gl}_3$ is subtracted by $(-1,1,0)$, and the highest weight with respect to the second summand $\mathfrak{gl}_3$ remains unchanged. If $X'$ contains exactly one element from $\{-4,2,3\}$, then when $r_{max}$ is subtracted from the highest weight with respect to the first term $\mathfrak{gl}_3$, $(0,-1,1)$ is subtracted, and the highest weight with respect to the second term $\mathfrak{gl}_3$ remains unchanged.

The proposition is proved.
	
\end{proof}

\begin{remark}
	We have actually proven that when onr subtracts the vector $r_{max}$ from $\gamma_{max}$ with the highest $\mathfrak{gl}_3\oplus \mathfrak{gl}_3$-weight, the following happens. Either it does not change, or the vector $(-1,1,0)$ or $(0,-1,1)$ is subtracted from one of its halves.
\end{remark}




\subsection{Formula of type \eqref{fv} for $\mathfrak{g}_2$}

Introduce the maximization procedure with respect to the considered subalgebra $\mathfrak{sl}_3\subset \mathfrak{g}_2$. This is done exactly as in section \ref{mxm} for $\mathfrak{gl}_3\oplus \mathfrak{gl}_3$. These two procedures are actually the same. This directly follows from the formulas \eqref{sl3g2} for the positive root generators of the subalgebra $\mathfrak{sl}_3\subset \mathfrak{g}_2$ under consideration. Therefore, we will continue to use the notation $max$ for the maximization procedure over $\mathfrak{gl}_3\oplus \mathfrak{gl}_3$ or, equivalently, over $\mathfrak{sl}_3$.

There is an analogue of Proposition \ref{mxmf} (the proof is repeated literally).

Next, let $V^{\mathfrak{sl}_3}_{(\gamma-sr)_{max}}$ be the $\mathfrak{sl}_3$ representation in the functional realization generated by $\mathcal{F}^{\mathfrak{g}_2}_{\gamma_{max}}(a)$.
There is an analogue of formula \eqref{fvff} for the chain $\mathfrak{g}_2\supset \mathfrak{sl}_3$.

\begin{lemma}
	\begin{equation}
	\label{lm}
	\mathcal{F}^{\mathfrak{g}_2}_{\gamma}(A)\in \bigoplus_{s\in \mathbb{Z}^k_{\geq 0}} V^{\mathfrak{sl}_3}_{(\gamma-sr)_{max}}
	\end{equation}
\end{lemma}

\begin{proof}

	First, we prove an analogue of formula \eqref{efd}. Specifically, if an element $X\in\mathfrak{g}_2$ is represented as $Z=\sum c_{i,j}E_{i,j}$, then
	
	\begin{equation}
	\label{eqdfg2}
	Z\mathcal{F}^{\mathfrak{g}_2}(a)=\sum c_{i,j}\sum_{X}\sum_{s\in \mathbb{Z}^k_{\geq 0}}c_{X,s}\mathcal{F}^{\mathfrak{g}_2}_{\gamma-e_{j,X}+e_{i,X}+sr}.
	\end{equation}
	
The proof of formula \eqref{efd} in \cite{a1} is based on the Basic Lemma (see Lemma 6.2 in \cite{a1}). An analogue of the Basic Lemma also holds in the case of $\mathfrak{g}_2$, and its proof in this case is identical to the proof of Lemma 6.2 in \cite{a1}. After that, the derivation of \eqref{eqdfg2} is identical to the derivation of \eqref{efd} in \cite{a1} (see the derivation of formula (6.7) in \cite{a1}). 
 
 Once we have formula \eqref{eqdfg2}, the proof of \eqref{lm} is identical to the derivation of \eqref{fvff}.
	
\end{proof}

\begin{prop}
	\label{pro5}
The representations on the right-hand side of formula \eqref{lm} are orthogonal with respect to the scalar product \eqref{skp} in the A-GKZ realization.
\end{prop}

\begin{proof}Note that from the explicit formula for $\mathcal{F}_{\gamma}^{\mathfrak{g}_2}$, we obtain
	
	\begin{equation}
		\label{fgg2}
	\mathcal{F}_{\gamma}^{\mathfrak{g}_2}=\sum_{w\in W} \pm \mathcal{F}_{\gamma+w},
	\end{equation}
	
where $W$ is the linear span of the second and third types of generators of the lattice $B_{GC}^{\mathfrak{g}_2}$, that is, in fact, $W=B_{GC}^{\mathfrak{g}_2}/B_{GC}^{\mathfrak{gl}_8}$. The second and third types of generators are constructed using the second and third types of equations in the A-GKZ system for the algebra $\mathfrak{g}_2$. These equations are based on relations that are invariant under the action of $\mathfrak{g}_2$. Indeed, the second-type relations are invariant even under the action of the algebra $\mathfrak{o}_8$, and the third-type relations are directly constructed from the invariants of the algebra $\mathfrak{g}_2$. It follows that for $w\in W$, we have $w_{max}=w$.

It follows that $\mathcal{F}_{\gamma_{max}+w}^{}=\mathcal{F}_{(\gamma+w)_{max}}^{}$. Now, using \eqref{fgg2}, the inclusion $\mathfrak{sl}_3\subset g$, and the definition of the representations $V_{\gamma_{max}}^g$ and $V_{\gamma_{max}}^{\mathfrak{g}_2}$, we have

	\begin{equation}
	\label{vvv}
	V_{\gamma_{max}}^{\mathfrak{g}_2}\subset \bigoplus_{w\in W}
	V_{\gamma_{max}+w}^{g}.
	\end{equation}
	
Let's check that the representations $V_{(\gamma-sr)_{max}}^{\mathfrak{g}_2}=V_{\gamma_{max}-sr_{max}}^{\mathfrak{g}_2}$ and $V_{(\gamma_{}-s'r)_{max}}^{\mathfrak{g}_2}=V_{\gamma_{max}-s'r_{max}}^{\mathfrak{g}_2}$ have different $\mathfrak{sl}_3$-highest weights when $s\neq s'$.
 Indeed, the first representation is contained in \begin{equation}\label{ur1}\bigoplus_{w\in W}
 V_{\gamma_{max}-sr_{max}+w}^{g},\end{equation} and the second is contained in \begin{equation}\label{ur2}\bigoplus_{w\in W}
 V_{\gamma_{max}-s'r_{max}+w}^{g}.\end{equation}
	
Consider the $\mathfrak{gl}_3\oplus \mathfrak{gl}_3$-highest weights of the individual terms in these sums. In the proof of Proposition \ref{pdbr}, it was shown that when subtracting vectors of the form $r_{max}\neq 0$, the highest weight with respect to $\mathfrak{gl}_3\oplus \mathfrak{gl}_3$ changes in such a way that the highest weight changes only with respect to one of the terms $\mathfrak{gl}_3$, while it remains unchanged with respect to the other term. At the same time, the corresponding $ \mathfrak{gl}_3$-highest weight changes in the following way: the vector $(-1,1,0)$ or $(0,-1,1)$ is subtracted from it. It can be easily verified by explicit calculations that the $\mathfrak{sl}_3$-highest weight must also change. Moreover, when subtracting $sr_{max}$ and $s'r_{max}$, if the condition $sr_{max} \neq s'r_{max}$ is satisfied, the resulting $\mathfrak{sl}_3$-highest weights are different. This means that $V_{\gamma_{max}-sr_{max}+w}^{g}$ and $V_{\gamma_{max}-s'r_{max}+w}^{g}$ have {\it different} $\mathfrak{sl}_3$-highest weights.
	
Next, we note that adding the vector $w$ does not change the $\mathfrak{sl}_3$-highest weights. This means that all terms in each of the sums \eqref{ur1} and \eqref{ur2} have the same $\mathfrak{sl}_3$-highest weight. So, the $\mathfrak{sl}_3$-higher weights $V_{\gamma_{max}-sr_{max}}^{\mathfrak{g}_2}$ and $V_{\gamma_{max}-s'r_{max}}^{\mathfrak{g}_2}$ are {\it different}.

But then the different terms in \eqref{lm} are orthogonal as eigenspaces corresponding to different eigenvalues for the self-adjoint Casimir operator.

\end{proof}

The same arguments as in the proof of \eqref{fv} give the equality

\begin{equation}
	\label{ffff}
F^{\mathfrak{g}_2}_{\gamma}(A)\in\bigoplus_{s\in \mathbb{Z}_{\geq 0}^k} V^{\mathfrak{sl}_3}_{(\gamma-sr)_{max}+w},
\end{equation}

\endproof

\subsection{Relation of the basis $F_{\gamma}^{\mathfrak{g}_2}(A)$ to the Gelfand-Tsetlin type basis for $\mathfrak{g}_2$}

Proposition \ref{pro5} says that the representations on the right in formula \eqref{ffff} are orthogonal with respect to the scalar product \eqref{skp} in the A-GKZ realization

Consider the order 

\begin{equation}
\label{por}
\gamma\prec \delta \Leftrightarrow \gamma+sr=\delta \,\,\, modB_{GC}^{\mathfrak{g}_2}.
\end{equation} From the formula \eqref{ffff} and the fact that the terms on the right-hand side of this formula are orthogonal, we obtain the following statement.

\begin{theorem}
 The orthogonalization with respect to the order \eqref{por} of the basis $F_{\gamma}^{\mathfrak{g}_2}(A)$ is the Gelfand-Tsetlin basis for the chain $\mathfrak{g}_2\supset \mathfrak{sl}_3$.
\end{theorem}

\end{document}